\documentclass[12pt,reqno]{article}

\usepackage{amssymb,amsmath,latexsym,amsthm,hyperref}
\usepackage{color,xcolor}
\usepackage{cite}
\usepackage{graphicx}
\usepackage{bbm}
\usepackage[font=scriptsize]{caption}

\newtheorem{theorem}{Theorem}

\newtheorem{proposition}[theorem]{Proposition}
\newtheorem{conjecture}[theorem]{Conjecture}
\newtheorem{lemma}[theorem]{Lemma}
\newtheorem{corollary}[theorem]{Corollary}

\graphicspath{ {images/} }
\newcommand{\PP}{\mathbb{P}}

\newcommand{\N}{\mathbb{N}}
\newcommand{\R}{\mathbb{R}}

\newcommand{\argmin}{\operatorname*{arg\,min}}

\newcommand{\conv}{\textup{conv}}

\renewcommand{\indent}{\hspace*{24pt}\ignorespaces}
\let\oldhref\href
\renewcommand{\href}[2]{\oldhref{#1}{\bfseries#2}}
\begin{document}
\title{The distribution of the number of distinct values in a finite exchangeable sequence}
\author{Theodore Zhu\\ University of California, Berkeley}
\date{\today}

\maketitle

\begin{abstract}
\noindent Let $K_n$ denote the number of distinct values among the first $n$ terms of an infinite exchangeable sequence of random variables $(X_1,X_2,\ldots)$. We prove for $n=3$ that the extreme points of the convex set of all possible laws of $K_3$ are those derived from i.i.d. sampling from discrete uniform distributions and the limit case with $\PP(K_3=3)=1$, and offer a conjecture for larger $n$. We also consider variants of the problem for finite exchangeable sequences and exchangeable random partitions.
\end{abstract}

\section{Introduction}

For an infinite sequence of real-valued random variables $(X_1,X_2,\ldots)$, let
\begin{equation}
K_n=K_n(X_1,\ldots,X_n):=\#\{X_i:1\leq i\leq n\},
\end{equation}
the number of distinct values appearing in the first $n$ terms. This article focuses on the case in which the sequence $(X_1,X_2,\ldots)$ is \textit{exchangeable}, meaning that its distribution is invariant under finite permutations of the indices. It is a well-known and celebrated result of de Finetti that any infinite exchangeable sequence is a mixture of i.i.d. sequences.
We explore ideas related to the following central question:
\begin{quote}Given a probability distribution $(a_1,\ldots,a_n)$ on $[n]:=\{1,\ldots,n\}$, is there an infinite exchangeable sequence of random variables $(X_1,X_2,\ldots)$ such that $\PP(K_n=k)=a_k$ for $1\leq k\leq n$?\end{quote}
The functional $K_n$ has been studied extensively in the context of the \textit{occupancy problem} as well as other closely related formulations including the birthday problem, the coupon collector's problem, and random partition structures \cite{MR0228020,MR0216548,MR2245368}. Much of the literature pertains to the asymptotic behavior of $K_n$ in the classical version in which the $X_i$ are i.i.d. discrete uniform random variables, as well as the general i.i.d. case. See \cite{MR2318403} for a recent survey with many references. Asymptotics of $K_n$ have also been studied for a random walk $(X_1,X_2,\ldots)$ with stationary increments \cite{MR0388547},\cite[Section 7.3]{MR2722836}. \\

Let us first consider the problem for small values of $n$. For $n=1$, the random variable $K_1$ is just the constant $1$. Next, it is easy to see that any probability distribution on $\{1,2\}$ can be achieved as the law of $K_2$ for some exchangeable sequence; indeed, for any $a\in[0,1]$, i.i.d. sampling from a distribution with a single atom having weight $\sqrt{a}$ yields $\PP(K_2=1)=a$. However, the problem is not trivial for $n=3$, as evident by the following bound due to Jim Pitman (proof in Section 3.)

\begin{proposition} \label{ofbound}
For $K_3$ the number of distinct values in the first $3$ terms of an infinite exchangeable sequence of random variables $(X_1,X_2,\ldots)$, 
\begin{equation}
\PP(K_3=2)\leq\frac{3}{4}.
\end{equation}
\end{proposition} 

Here we present the main open problem and result of this article. Let $\boldsymbol{v}_{n,m}$ denote the law of $K_{n,m}:=K_n(X_{m,1},\ldots,X_{m,n})$ where $X_{m,i}$ are i.i.d. with uniform distribution on $m$ elements, i.e. 
\begin{equation}
\boldsymbol{v}_{n,m}=\big(\PP(K_{n,m}=k):1\leq k\leq n\big)
\end{equation} 
and let $\boldsymbol{v}_{n,\infty}=(0,\ldots,0,1)$, corresponding to the limit case $m=\infty$. Let
\begin{equation}
V_n:=\{\boldsymbol{v}_{n,m}:m=1,2,\ldots,\infty\}
\end{equation} 
and let $\conv(V_n)$ denote the convex hull of $V_n$.

\begin{conjecture} \label{mainconj} For $n\geq 3$,
\begin{enumerate}
\item[(i)] The set of extreme points of $\conv(V_n)$ is $V_n$.
\item[(ii)] The set of possible laws of $K_n$ for an infinite exchangeable sequence $(X_1,X_2,\ldots)$ is $\conv(V_n)$. 
\end{enumerate}
\end{conjecture}
\vspace{0.3cm}
\begin{theorem} \label{mainthm}
Assertions $(i)$ and $(ii)$ are true for $n=3$.
\end{theorem}
\vspace{0.3cm}

The rest of this article is organized as follows. Section \ref{prelims} establishes notation and the fundamentals of our approach. Section \ref{k3main} covers some properties of the law of $K_3$ leading to a proof of Theorem \ref{mainthm}, and Section \ref{highdim} aims to extend some of these results to $K_n$ for larger $n$. Section \ref{finitesection} considers a variant of the main problem for finite exchangeable sequences by appealing to the framework of exchangeable random partitions, and Section \ref{twoparamsection} explores a remarkable symmetry for $K_3$ in the Ewens-Pitman two-parameter partition model.

\section{Preliminaries} \label{prelims}
For an i.i.d. sequence $(X_1,X_2,\ldots)$, there is an associated \textit{ranked discrete distribution} $(p_1,p_2,\ldots)$ with $p_1\geq p_2\geq\ldots\geq 0$ and $\sum_{i=1}^\infty p_i\leq 1$ where the $p_i$ are the weights of the atoms for the law of $X_i$ in decreasing order, and $1-\sum_{i=1}^\infty p_i$ is the weight of the continuous component.\\

Consider the set
\begin{equation}
\nabla_\infty:=\Big\{(p_1,p_2,\ldots):p_1\geq p_2\geq\ldots\geq 0,\sum_{i=1}^\infty p_i\leq 1\Big\},
\end{equation}
sometimes referred to as the infinite dimensional 
\textit{Kingman simplex} as in \cite{MR2596654}. The uniform distribution on $m$ elements corresponds to
\begin{equation}
\boldsymbol{u}_m:=\Big(\underbrace{\frac{1}{m},\ldots,\frac{1}{m}}_{m\text{ times}},0,0,\ldots\Big)\in\nabla_\infty.
\end{equation}
and any non-atomic law corresponds to $\boldsymbol{u}_\infty:=(0,0,\ldots)\in\nabla_\infty$. With Conjecture \ref{mainconj} and Theorem \ref{mainthm} in mind, note that
\begin{equation}
\big\{\boldsymbol{u}_m:m=1,2,\ldots,\infty\big\}
\end{equation}
is precisely the set of extreme points of $\nabla_\infty$ \cite[Theorem 4.1]{MR954608}. Any $(p_1,p_2,\ldots)\in\nabla_\infty$ has a unique representation as a convex combination of $\boldsymbol{u}_m$, $m=1,2,\ldots,\infty$ given by 
\begin{equation}
(p_1,p_2,\ldots)=p_*\boldsymbol{u}_\infty+\sum_{i=1}^\infty(p_i-p_{i+1})\boldsymbol{u}_i,\qquad p_*=1-\sum_{i=1}^\infty p_i.
\end{equation}
This is a discrete version of Khintchine's representation theorem for unimodal distributions \cite{khintchine1938unimodal}.\\ 

It is easy to see that the law of $K_n$ for an i.i.d sequence depends only on the ranked frequencies of the atoms. For example, let
\begin{equation}
q_{n,i}(p_1,p_2,\ldots):=\PP\big(K_n=i\big)
\end{equation}
where $K_n=K_n(X_1,\ldots,X_n)$ for i.i.d. $X_i$ with ranked frequencies $(p_1,p_2,\ldots)$.  
Then for $n=3$,
\begin{align}
q_{3,1}(p_1,p_2,\ldots)&=\sum_{i=1}^\infty p_i^3 \label{q1}\\
q_{3,2}(p_1,p_2,\ldots)&=\sum_{i=1}^\infty 3p_i^2(1-p_i) \label{q2}\\
q_{3,3}(p_1,p_2,\ldots)&=1-\sum_{i=1}^\infty\big[3p_i^2-2p_i^3\big].
\end{align} 

For the general exchangeable case, de Finetti's theorem guarantees that the law of $K_n$ for an exchangeable sequence of random variables $(X_1,X_2,\ldots)$ is a convex combination of laws of $K_n$ for i.i.d. sequences. This property allows us to focus on the i.i.d. case and the simplification to ranked discrete distributions. \\

Note that there is an equivalent reformulation of the problem in the setting of exchangeable random partitions; see e.g. \cite{MR2245368} for relevant background on the subject. For an exchangeable random partition $\Pi=(\Pi_n)$ of $\N$, let $K_n$ denote the number of \textit{clusters} in the restriction $\Pi_n$ of $\Pi$ to $[n]$. Through \textit{Kingman's respresentation theorem} \cite{MR509954} for exchangeable random partitions of $\N$ in terms of random ranked discrete distributions, the possible laws of $K_n$ in this setting are identical to the possible laws of $K_n$ as defined originally in this paper as the number of distinct values in the first $n$ terms of an exchangeable sequence $(X_1,X_2,\ldots)$. In Sections 5 -- 7, we explore some related problems in the framework of exchangeable random partitions.\\\\\\
\textbf{Notations and conventions.} If a ranked discrete distribution $(p_1,p_2,\ldots)$ has finitely many atoms, i.e. there exists $m$ such that $p_i=0$ for all $i>m$, we call it a \textit{finite} distribution and abbreviate it as $(p_1,\ldots,p_m)$ when convenient. Since all of the functionals that we work with on $\nabla_\infty$ are symmetric functions of the arguments, we understand an equivalence between an unordered discrete distribution $(p_1,p_2,\ldots)$ and its ranked version. Unless otherwise stated, it is implicit in the appearance of $(p_1,p_2,\ldots)$ or $(p_1,\ldots,p_m)$ that the conditions $p_i\geq 0$ and $\sum p_i\leq 1$ hold.

\section{Laws of $K_3$} \label{k3main}
To simplify notation in this section, let
\begin{equation} 
q_i:=q_{3,i}=\PP(K_3=i)
\end{equation}
where $q_i$ may be treated as a functional on $\nabla_\infty$.

\begin{lemma} \label{merge}
For $(p_1,\ldots,p_m)$ with $m\geq 3$ and $p_1\leq\ldots\leq p_m$, 
\begin{equation}
q_2(p_1+p_2,p_3\ldots,p_m)\geq q_2(p_1,p_2,p_3,\ldots,p_m).
\end{equation}
\end{lemma}
\begin{proof}
Let $a=p_1$ and $b=p_2$. We have
\begin{equation}
q_2(a,b,p_3,\ldots,p_m)=3a^2(1-a)+3b^2(1-b)+\sum_{i=1}^m 3p_i^2(1-p_i)
\end{equation}
and
\begin{equation}
q_2(a+b,p_3,\ldots,p_m)=3(a+b)^2(1-a-b)+\sum_{i=1}^m 3p_i^2(1-p_i).
\end{equation}
Then 
\begin{align}
q_2(a+b,p_3,\ldots,p_m)-q_2(a,b,p_3,\ldots,p_m)&=3(a+b)^2(1-a-b)-3a^2(1-a)-3b^2(1-b)\\ 
&=6ab(1-a-b)-3a^2b-3ab^2\\
&=3ab(2-3(a+b))\\
&\geq 0
\end{align}
since $a$ and $b$ are the two smallest values among $\{a,b,p_3,\ldots,p_m\}$ so $a+b\leq\frac{2}{m}\leq\frac{2}{3}$ for $m\geq 3$.\\
\end{proof}
This shows that for any $(p_1,\ldots,p_m)$ with $m\geq 3$, merging the two smallest values among $\{p_1,\ldots,p_m\}$ does not decrease $q_2$.\\

\begin{proof}[Proof of Proposition \ref{ofbound}]
By convexity, it suffices to prove the inequality for i.i.d. sequences. Since 
\begin{equation}
q_2(p_1,p_2,\ldots)=\sum_{i=1}^\infty 3p_i^2(1-p_i)=\lim_{m\rightarrow\infty}\sum_{i=1}^m 3p_i^2(1-p_i)=\lim_{m\rightarrow\infty}q_2(p_1,\ldots,p_m), 
\end{equation}
it is enough to establish the inequality $q_2(p_1,\ldots,p_m)\leq\frac{3}{4}$ for finite discrete distributions $(p_1,\ldots,p_m)$. If $m=2$, then $q_2(p_1,p_2)=3p_1^2(1-p_1)+3p_2^2(1-p_2)$ which attains its maximum value of $\frac{3}{4}$ subject to $p_1,p_2\geq 0$ and $p_1+p_2\leq 1$ at $p_1=p_2=\frac{1}{2}$. For $m\geq 3$, by Lemma \ref{merge} repeatedly merging the two smallest values until no more than two nonzero values remain gives $q_2(p_1,\ldots,p_m)\leq q_2(\frac{1}{2},\frac{1}{2})=\frac{3}{4}$.\\
\end{proof}

Consider the law of $K_3$ for an i.i.d. sequence $(X_1,X_2,\ldots)$ where each $X_i$ has the uniform distribution $\boldsymbol{u}_N:=(\frac{1}{N},\ldots,\frac{1}{N})$. A probability distribution $(q_1,q_2,q_3)$ of $K_3$ (on $\{1,2,3\}$) can be represented by any pair of its coordinates; here we shall work with $(q_1,q_3):=\big(\PP(K_3=1),\PP(K_3=3)\big)$. Then
\begin{align}
q_1(\boldsymbol{u}_N)&:=\PP(K_3(\boldsymbol{u}_N)=1)=\frac{1}{N^2}\\
q_3(\boldsymbol{u}_N)&:=\PP(K_3(\boldsymbol{u}_N)=3)=\frac{(N-1)(N-2)}{N^2}.
\end{align} 
The set of points $\{\boldsymbol{v}_N:N\in\N\}=\{(1,0),(\tfrac{1}{4},0),(\tfrac{1}{9},\tfrac{2}{9}),(\tfrac{1}{16},\tfrac{6}{16}),(\tfrac{1}{25},\tfrac{12}{25}),(\tfrac{1}{36},\tfrac{20}{36}),\ldots\}$ where 
\begin{equation} \label{vndef}
\boldsymbol{v}_N:=(q_1(\boldsymbol{u}_N),q_3(\boldsymbol{u}_N))=\Big(\frac{1}{N^2},\frac{(N-1)(N-2)}{N^2}\Big)
\end{equation}
 are shown in Figures \ref{segments} and \ref{map345}, with line segments connecting consecutive points.\\\\
\begin{figure}[h!]
\centering
\begin{minipage}{0.5\textwidth}
  \centering
  \includegraphics[width=1\linewidth]{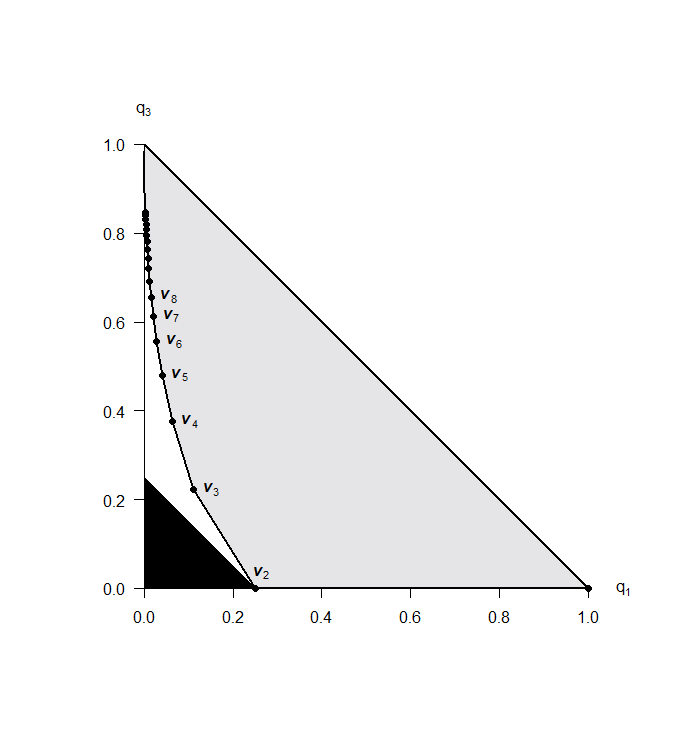}
  \captionsetup{width=0.8\linewidth}
  \vspace{-1.7cm}
  \captionof{figure}{Probability distributions of $K_3$ represented as points $(q_1,q_3)=\big(\PP(K_3=1),\PP(K_3=3)\big)$ with $q_1$ horizontal and $q_3$ vertical. Shaded in black is the restricted region specified by Proposition \ref{ofbound}. The gray region is the closed convex hull of $\{\boldsymbol{v}_N:N\in\N\}$ where $\boldsymbol{v}_N$ corresponds to the distribution of $K_3$ for i.i.d. sampling from a discrete uniform distribution on $N$ elements, as defined in \eqref{vndef}.}
  \label{segments}
\end{minipage}%
\begin{minipage}{0.5\textwidth}
  \centering
  \includegraphics[width=1\linewidth]{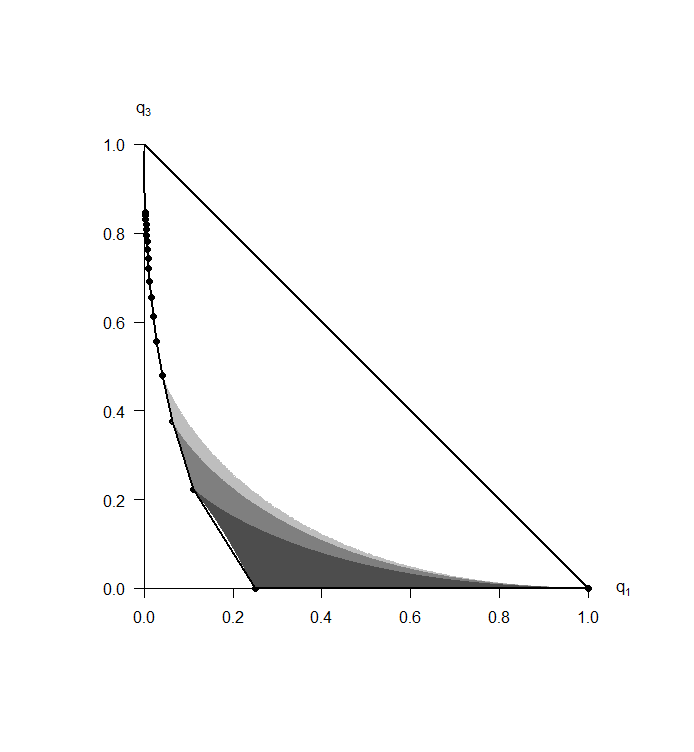}
  \vspace{-1.7cm}
  \captionsetup{width=0.81\linewidth}
  \captionof{figure}{The shaded regions (nested) correspond to the images of $\{(p_1,\ldots,p_m):p_i\geq 0,\sum p_i=1\}$ under the map $(p_1,\ldots,p_m)\mapsto\big(q_1(p_1,\ldots,p_m),q_3(p_1,\ldots,p_m)\big)$ for $m=3$ (dark), $m=4$ (dark and medium), and $m=5$ (dark, medium, and light). The existence of the gap between the left boundary of the dark region and the line segment connecting $\boldsymbol{v}_2$ and $\boldsymbol{v}_3$ is a consequence of Lemme \ref{buniform}}
  \label{map345}
\end{minipage}
\end{figure}

The slope of the line connecting $\boldsymbol{v}_N=\big(\tfrac{1}{N^2},\tfrac{(N-1)(N-2)}{N^2}\big)$ and $\boldsymbol{v}_{N+1}=\big(\tfrac{1}{(N+1)^2},\tfrac{N(N-1)}{(N+1)^2}\big)$ is 
\begin{equation} \label{Nslope}
\frac{\frac{N(N-1)}{(N+1)^2}-\frac{(N-1)(N-2)}{N^2}}{\frac{1}{(N+1)^2}-\frac{1}{N^2}}=-\frac{(N-1)(3N+2)}{2N+1};
\end{equation}
this is increasing in $N$ which proves Theorem \ref{mainthm}(i).
The equation of the $N$th line is given by 
\begin{equation}
	q_3-\frac{(N-1)(N-2)}{N^2}=-\frac{(N-1)(3N+2)}{2N+1}\bigg(q_1-\frac{1}{N^2}\bigg)
\end{equation}
or after rearranging,
\begin{equation} \label{lineN}
q_3+\frac{(N-1)(3N+2)}{2N+1}q_1=\frac{2N-2}{2N+1}.
\end{equation}
For $\boldsymbol{p}=(p_1,\ldots,p_m)$, define according to the left-hand side of \eqref{lineN} the functional
\begin{equation}
L_N(\boldsymbol{p}):=q_3(\boldsymbol{p})+\frac{(N-1)(3N+2)}{2N+1}q_1(\boldsymbol{p})
\end{equation}
which may be reexpressed as
\begin{align}
L_N(\boldsymbol{p})&=1-\big(1-L_N(\boldsymbol{p})\big)\\
&=1-\bigg(1-q_3(\boldsymbol{p})-q_1(\boldsymbol{p})-\bigg[\frac{(N-1)(3N+2)}{2N+1}-1\bigg]q_1(\boldsymbol{p})\bigg)\\
&=1-q_2(\boldsymbol{p})+\frac{3(N^2-N-1)}{2N+1}q_1(\boldsymbol{p})\\
&=1-\sum_{i=1}^m 3p_i^2(1-p_i)+\frac{3(N^2-N-1)}{2N+1}\sum_{i=1}^m p_i^3\\
&=1-3\sum_{i=1}^m p_i^2+\frac{3N(N+1)}{2N+1}\sum_{i=1}^m p_i^3.
\end{align}
Define
\begin{equation} \label{fformula}
f(N):=\frac{3N(N+1)}{2N+1}
\end{equation}
so 
\begin{equation}
L_N(\boldsymbol{p})=1-3\sum_{i=1}^m p_i^2+f(N)\sum_{i=1}^m p_i^3.
\end{equation}
To better understand the sequence of values $f(N)$, note that $f$ is increasing and 
\begin{equation}
	N<\frac{2N+2}{2N+1}(N)=\frac{2}{3}\cdot\underbrace{\frac{3N(N+1)}{2N+1}}_{f(N)}=\frac{2N}{2N+1}(N+1)<N+1.
\end{equation}
The first few values are $f(1)=2$, $f(2)=\frac{18}{5}$, $f(3)=\frac{36}{7}$, $f(4)=\frac{60}{9}$.
\begin{lemma} \label{linebounds} 
For $N\geq 1$ and any $\boldsymbol{p}=(p_1,\ldots,p_m)$ with $p_1\geq\ldots\geq p_m\geq 0$ and $\sum p_i\leq 1$,
\begin{equation}
L_N(\boldsymbol{p})\geq\frac{2N-2}{2N+1}.
\end{equation}
\end{lemma}
\vspace{.5cm}
Geometrically, Lemma \ref{linebounds} asserts that for any $\boldsymbol{p}=(p_1,\ldots,p_m)$, the point $\big(q_1(\boldsymbol{p}),q_3(\boldsymbol{p})\big)$ lies on or above each of the lines connecting $\boldsymbol{v}_N$ and $\boldsymbol{v}_{N+1}$ for $N\in\N$. It will be shown in the proof that for $N\geq 2$, $L_N(\boldsymbol{p})=\frac{2N-2}{2N+1}$ if and only if $\boldsymbol{p}=\boldsymbol{u}_N$ or $\boldsymbol{p}=\boldsymbol{u}_{N+1}$; as for $N=1$, $L_1(\boldsymbol{p})=q_3(\boldsymbol{p})=0$ is attained if and only if $\boldsymbol{p}=(p_1,p_2)$ with $p_1+p_2=1$.  \\

\indent The strategy for proving Lemma \ref{linebounds} is to show that $L_N$ is minimized at precisely $\boldsymbol{v}_N$ and $\boldsymbol{v}_{N+1}$ by reducing the domain of minimization in stages, first to $(p_1,\ldots,p_m)$ with $\sum p_i=1$, then to the uniform distributions, and finally to $\boldsymbol{u}_N$ and $\boldsymbol{u}_{N+1}$. The key to the proof is the following \textit{merging} lemma, which generalizes Lemma \ref{merge}.
\begin{lemma} \label{genmerge}
For $N\geq 1$ and $(p_1,\ldots,p_m)$ with $m\geq 2$,
\begin{equation}
L_N(p_1+p_2,p_3,\ldots,p_m)-L_N(p_1,p_2,p_3,\ldots,p_m)=3p_1p_2\big[(p_1+p_2)f(N)-2\big]
\end{equation}
which is positive, negative, or zero according to the sign of $p_1+p_2-\tfrac{2}{f(N)}$.
\end{lemma}
\begin{proof}
Let $a=p_1$ and $b=p_2$. We have
\begin{equation}
L_N(a,b,p_3,\ldots,p_m)=1-3a^2-3b^2-3\sum_{i=3}^m p_i^2+f(N)(a^3+b^3)+f(N)\sum_{i=3}^m p_i^3
\end{equation}
and
\begin{equation}
L_N(a+b,p_3,\ldots,p_m)=1-3(a+b)^2-3\sum_{i=3}^m p_i^2+f(N)(a+b)^3+f(N)\sum_{i=3}^m p_i^3.
\end{equation}
Then
\begin{align}
	L_N(a+b,p_3,\ldots,p_m)-L_N(a,b,p_3,\ldots,p_m)&=-6ab+f(N)(3a^2b+3ab^2)\\
	&=3ab\big[(a+b)f(N)-2\big].
\end{align}
\end{proof}
The proof of Lemma \ref{linebounds} is organized according to the following lemmas.

\begin{lemma} \label{pdbest}
Let $\mathcal{P}$ denote the set of all finite ranked discrete distributions, and let $\mathcal{P}^1$ denote the set of finite ranked discrete distributions $(p_1,\ldots,p_m)$ with $\sum p_i=1$. Then for any $N\geq 1$, we have the equality of sets
\begin{equation}
	\argmin_{\boldsymbol{p}\in\mathcal{P}}L_N(\boldsymbol{p})=\argmin_{\boldsymbol{p}\in\mathcal{P}^1}L_N(\boldsymbol{p})
\end{equation}
\end{lemma}
\begin{proof}
Let $\boldsymbol{p}_0=(p_1,\ldots,p_m)\in\mathcal{P}$ such that $\sum_{i=1}^m p_i<1$. Let $\varepsilon$ satisfy $0<\varepsilon<\min\{\frac{3}{f(N)},1-\sum_{i=1}^m p_i\}$. Then
\begin{align}
	L_N(\varepsilon,p_1,\ldots,p_m)&=-3\varepsilon^2+f(N)\varepsilon^3+L_N(p_1,\ldots,p_m)\\
	&=\varepsilon^2(f(N)\varepsilon-3)+L_N(p_1,\ldots,p_m)\\
	&<L_N(p_1,\ldots,p_m).
\end{align}
This shows that if $\boldsymbol{p}_0\notin\mathcal{P}^1$, then $\boldsymbol{p}_0\notin\argmin_{\boldsymbol{p}\in\mathcal{P}}L_N(\boldsymbol{p})$.\\
\end{proof}

\begin{lemma} \label{onlyunif}
Let $\mathcal{P}^1$ denote the set of finite ranked discrete distributions $(p_1,\ldots,p_m)$ with $\sum p_i=1$, and let $\mathcal{U}:=\big\{\boldsymbol{u}_m:m\in\N\}$. Then for $N\geq 2$, we have the equality of sets
\begin{equation}
\argmin_{\boldsymbol{p}\in\mathcal{P}^1}L_N(\boldsymbol{p})=\argmin_{\boldsymbol{p}\in\mathcal{U}} L_N(\boldsymbol{p})
\end{equation}
\end{lemma}
\begin{proof}
Let $\boldsymbol{p}_0=(p_1,\ldots,p_m)$, not necessarily ranked, such that $\sum_{i=1}^m p_i=1$. Suppose $\boldsymbol{p}_0$ has a pair of distinct nonzero values, say $a=p_1$ and $b=p_2$ with $a,b>0$ and $a\neq b$. Consider the three cases as designated in Lemma \ref{genmerge}, noting that $\tfrac{2}{f(N)}<1$ for $N\geq 2$. \\\\
\indent (i)\ \ If $a+b<\frac{2}{f(N)}$, then $L_N(a+b,p_3,\ldots,p_m)<L_N(a,b,p_3,\ldots,p_m)$ by Lemma \ref{genmerge}.\\
\indent (ii) If $a+b>\frac{2}{f(N)}$,
\begin{align}
	L_N(\tfrac{a+b}{2}&,\tfrac{a+b}{2},p_3,\ldots,p_m)-L_N(a,b,p_3,\ldots,p_m)\\
	&=\big(L_N(a+b,p_3,\ldots,p_m)-L_N(a,b,p_3,\ldots,p_m)\big)\\ &\indent-\big(L_N(a+b,p_3,\ldots,p_m)-L_N(\tfrac{a+b}{2},\tfrac{a+b}{2},p_3,\ldots,p_m)\big)\nonumber\\
	&=3ab\big((a+b)f(N)-2\big)-3(\tfrac{a+b}{2})^2\big((a+b)f(N)-2\big)\\
	&=3\big(ab-(\tfrac{a+b}{2})^2\big)\big((a+b)f(N)-2\big)
\end{align}
\indent\indent which is negative since $ab-(\frac{a+b}{2})^2<0$ and $(a+b)f(N)-2>0$.\\
\indent (iii) If $a+b=\frac{2}{f(N)}<1$, then there must exist a third nonzero value, say $p_3=c>0$. If \newline\indent\indent $c=\frac{2}{f(N)}$, then $a\neq c$ and $a+c>\frac{2}{f(N)}$ so $L_N(\frac{a+c}{2},\frac{a+c}{2},b,p_4,\ldots,p_m)<L_N(a,b,c,p_4,\ldots,p_m)$ \newline\indent\indent by case (ii). If $c\neq\frac{2}{f(N)}$, then by merging $a$ and $b$, which does not change $L_N$, and then \newline\indent\indent subsequently averaging $a+b$ and $c$ gives $L_N(\frac{a+b+c}{2},\frac{a+b+c}{2},p_4,\ldots,p_m)<L_N(a,b,c,p_4,\ldots,p_m)$\newline\indent\indent by case (ii) again.\\\\
Since permuting values in any discrete distribution does not change $L_N$, the analysis above holds for all ranked discrete distributions and thus shows that among $\boldsymbol{p}\in\mathcal{P}^1$, $L_N$ cannot be minimized at any $\boldsymbol{p}$ with a pair of distinct nonzero values, i.e. any non-uniform distribution.\\
\end{proof}
\textbf{Remark.} As mentioned previously, for $N=1$, $$\argmin_{\boldsymbol{p}\in\mathcal{P}}L_1(\boldsymbol{p})=\{(p_1,p_2):p_1\geq p_2\geq 0, p_1+p_2=1\}$$ which differs from the general case $N\geq 2$. The reason the proof of Lemma \ref{onlyunif} fails for $N=1$ is that $f(1)=2$, so $\tfrac{2}{f(1)}=1$ and case (iii) of the proof breaks down.\\
\begin{lemma} \label{buniform}
Let $\mathcal{U}:=\{\boldsymbol{u}_m:m\in\N\}$. Then for $N\geq 1$,
\begin{equation}
	\argmin_{\boldsymbol{p}\in\mathcal{U}}L_N(\boldsymbol{p})=\{\boldsymbol{u}_N,\boldsymbol{u}_{N+1}\}
\end{equation}
\end{lemma}
\begin{proof}
The claim is obvious based on Figure \ref{segments}, which shows that the slopes between $\boldsymbol{v}_N$ and $\boldsymbol{v}_{N+1}$ for $N\in\N$ are decreasing in $N$. Indeed, the slope of the $N$th line segment is computed in \eqref{Nslope} as
\begin{equation}
-\frac{(N-1)(3N+2)}{2N+1}=-\frac{3N^2-N-2}{2N+1}=2-\frac{3N(N+1)}{2N+1}=2-f(N)
\end{equation}
which is decreasing in $N$.\\
\end{proof}

\begin{proof}[Proof of Lemma \ref{linebounds}]
The claim holds trivially for $N=1$. For $N\geq 2$, applying Lemmas \ref{pdbest}, \ref{onlyunif}, and \ref{buniform} yields
\begin{equation}
\argmin_{\boldsymbol{p}\in\mathcal{P}}L_N(\boldsymbol{p})=\argmin_{\boldsymbol{p}\in\mathcal{P}^1}L_N(\boldsymbol{p})=\argmin_{\boldsymbol{p}\in\mathcal{U}} L_N(\boldsymbol{p})=\{\boldsymbol{u}_N,\boldsymbol{u}_{N+1}\}
\end{equation}
and therefore for any $\boldsymbol{p}=(p_1,\ldots,p_m)$ with $p_i\geq 0$ and $\sum p_i\leq 1$, 
\begin{equation}
L_N(\boldsymbol{p})\geq L_N(\boldsymbol{u}_N)=L_N(\boldsymbol{u}_{N+1})=\frac{2N-2}{2N+1}.
\end{equation}
\end{proof}
\begin{proof}[Proof of Theorem \ref{mainthm}]
Part (i) was proven earlier by the slope computation \eqref{Nslope} and illustrated in Figure \ref{segments}. For part (ii), Lemma \ref{linebounds} asserts that $(q_1(\boldsymbol{p}),q_2(\boldsymbol{p}),q_3(\boldsymbol{p}))\in \conv(V_3)$ for any finite ranked discrete distribution $\boldsymbol{p}$. Extension to infinite discrete distributions $(p_1,p_2,\ldots)$ follows because $\lim_{m\rightarrow\infty}q_i(p_1,\ldots,p_m)=q_i(p_1,p_2,\ldots)$, and then extension to exchangeable sequences holds by convexity.\\
\end{proof}
\section{Higher dimensions} \label{highdim}
This section aims to extend some of the results in the previous section to $K_n$ for larger $n$. Here $q_{n,i}:=\PP(K_n=i)$. We begin by generalizing Lemma \ref{merge} and Proposition \ref{ofbound}.  
\begin{lemma} \label{nmerge}
For $n\geq 3$ and $(p_1,\ldots,p_m)$ with $m\geq 3$, $\sum_{i=1}^m p_i=1$, $p_1\leq\ldots\leq p_m$,
\begin{equation}
q_{n,2}(p_1+p_2,p_3,\ldots,p_m)\geq q_{n,2}(p_1,p_2,p_3,\ldots,p_m).
\end{equation}
\end{lemma}
\vspace{0.5cm}
The proof requires the following inequality:
\begin{lemma} \label{binomineq}
For $a,b>0$ and $n\geq 2$, 
\begin{equation}
4\big(\tfrac{n-1}{n}\big)ab(a+b)^{n-2}\leq(a+b)^n-a^n-b^n\leq nab(a+b)^{n-2}
\end{equation}
\end{lemma}
\begin{proof}
We have
\begin{equation}
(a+b)^n-a^n-b^n=\sum_{k=1}^{n-1}\binom{n}{k}a^{k}b^{n-k}=ab\sum_{k=0}^{n-2}\binom{n}{k+1}a^kb^{n-2-k}. \label{shiftb}
\end{equation}
Observe that
\begin{equation}
\binom{n}{k+1}=\frac{n(n-1)(n-2)!}{(k+1)k!(n-k-1)(n-k-2)!}=\frac{n(n-1)}{(k+1)(n-k-1)}\binom{n-2}{k};\\
\end{equation}
the denominator $(k+1)(n-k-1)$ is no greater than $(n/2)^2$, and is minimized at $k=0$ and $k=n-2$, so
\begin{equation} \label{nlb}
\binom{n}{k+1}\geq\frac{n(n-1)}{(n/2)^2}\binom{n-2}{k}=4\frac{n-1}{n}\binom{n-2}{k}
\end{equation}
and 
\begin{equation} \label{nub}
\binom{n}{k+1}\leq n\binom{n-2}{k}.
\end{equation}
The result follows by substituting inequalities \eqref{nlb} and \eqref{nub} into \eqref{shiftb} and appealing to the binomial theorem.\\
\end{proof}

\begin{proof}[Proof of Lemma \ref{nmerge}]
Let $a=p_1$ and $b=p_2$. We can compute $$q_{n,2}(a,b,p_3,\ldots,p_m)=\PP\big(K_n(a,b,p_3,\ldots,p_m)=2\big)$$ by conditioning on the appearance of the first two values:
\begin{equation}
q_{n,2}(a,b,p_3,\ldots,p_m)=\sum_{k=1}^{n-1}\binom{n}{k}a^k b^{n-k}+\sum_{k=1}^{n-1}\binom{n}{k}a^k\sum_{i=3}^m p_i^{n-k}+\sum_{k=1}^{n-1}\binom{n}{k}b^k\sum_{i=3}^m p_i^{n-k}+\sum_{3\leq i<j\leq m}\sum_{k=1}^{n-1}\binom{n}{k}p_i^kp_j^{n-k}.
\end{equation}
Note that the first term, which is an expression for the probability that the first two values both appear and are the only ones to appear in the first $n$ observations, is also equal to $(a+b)^n-a^n-b^n$. Similarly,
\begin{equation}
q_{n,2}(a+b,p_3,\ldots,p_m)=\sum_{k=1}^{n-1}\binom{n}{k}(a+b)^k\sum_{i=3}^m p_i^{n-k}+\sum_{3\leq i<j\leq m}\sum_{k=1}^{n-1}\binom{n}{k}p_i^kp_j^{n-k}.
\end{equation}
For $m\geq 3$, the difference after appropriate cancellations and then applying Lemma \ref{binomineq} is
\begin{align}
q_{n,2}&(a+b,p_3,\ldots,p_m)-q_{n,2}(a,b,p_3,\ldots,p_m)=\sum_{k=1}^{n-1}\binom{n}{k}\big[(a+b)^k-a^k-b^k\big]\sum_{i=3}^m p_i^{n-k}-\sum_{k=1}^{n-1}\binom{n}{k}a^k b^{n-k}\\
&=\underbrace{\sum_{k=1}^{n-2}\binom{n}{k}\big[(a+b)^k-a^k-b^k\big]\sum_{i=3}^m p_i^{n-k}}_{\geq 0}+n\underbrace{\big[(a+b)^{n-1}-a^{n-1}-b^{n-1}\big]}_{\geq 4(\frac{n-2}{n-1})ab(a+b)^{n-3}\geq 2ab(a+b)^{n-3}}\sum_{i=3}^m p_i-\underbrace{\big[(a+b)^n-a^n-b^n\big]}_{\leq nab(a+b)^{n-2}}\\
&\geq nab(a+b)^{n-3}\Big[2\sum_{i=3}^m p_i-(a+b)\Big].
\end{align}
Since $\sum_{i=1}^m p_i=1$ and $a\leq b\leq p_3\leq\ldots\leq p_m$, it follows that $\sum_{i=3}^m p_i\geq \frac{m-2}{m}$ and $a+b\leq\frac{2}{m}$, so
\begin{equation}
2\sum_{i=3}^m p_i-(a+b)\geq 2\Big(\frac{m-2}{m}\Big)-\frac{2}{m}=\frac{2(m-3)}{m}\geq 0
\end{equation}
and therefore merging the two smallest values among $\{p_1,\ldots,p_m\}$ does not decrease $q_{n,2}$ provided that there are at least 3 nonzero values.
\end{proof}
\begin{lemma} \label{fullbetter}
For any $(p_1,\ldots,p_m)$ and $n\geq 3$,
\begin{equation}
q_{n,2}(p_1,\ldots,p_m,p_*)\geq q_{n,2}(p_1,\ldots,p_m)
\end{equation}
where $p_*:=1-\sum_{i=1}^m p_i$.
\end{lemma}
\begin{proof}
We have 
\begin{equation}
q_{n,2}(p_1,\ldots,p_m)=\sum_{1\leq i<j\leq m}\sum_{k=1}^{n-1}\binom{n}{k}p_i^kp_j^{n-k}+\sum_{i=1}^m np_i^{n-1}p_*
\end{equation}
and
\begin{equation}
q_{n,2}(p_1,\ldots,p_m,p_*)=\sum_{1\leq i<j\leq m}\sum_{k=1}^{n-1}\binom{n}{k}p_i^kp_j^{n-k}+\sum_{i=1}^m\sum_{k=1}^{n-1}\binom{n}{k}p_i^{k}p_*^{n-k},
\end{equation}
so 
\begin{equation}
q_{n,2}(p_1,\ldots,p_m,p_*)-q_{n,2}(p_1,\ldots,p_m)=\sum_{i=1}^m\sum_{k=1}^{n-2}p_i^kp_*^{n-k}\geq 0.
\end{equation}
\end{proof}
\begin{theorem} \label{n2bound}
For any exchangeable sequence of random variables $(X_1,X_2,\ldots)$ and any $n\geq 3$,
\begin{equation}
\PP(K_n=2)\leq1-2^{-(n-1)}.
\end{equation}
\end{theorem}
\begin{proof}
As in the proof of Proposition \ref{ofbound}, it suffices to show that $q_{n,2}(p_1,\ldots,p_m)\leq 1-2^{-(n-1)}$ for any $(p_1,\ldots,p_m)$. If $m=2$ and $p_1+p_2=1$, then
\begin{equation}
q_{n,2}(p_1,p_2)=1-p_1^n-p_2^n
\end{equation}
which attains its maximum of $1-2^{-(n-1)}$ at $p_1=p_2=\tfrac{1}{2}$. For $m\geq 3$, by Lemmas \ref{nmerge} and \ref{fullbetter} we have
\begin{equation}
q_{n,2}(p_1,\ldots,p_m)\leq q_{n,2}(p_1,\ldots,p_m,p_*)\leq q_{n,2}\big(\tfrac{1}{2},\tfrac{1}{2}\big)=1-2^{-(n-1)}.
\end{equation}
\end{proof}
The difficulty in extending the proof of Theorem \ref{mainthm}(ii) to apply to Conjecture \ref{mainconj}(ii) is that there is no simple generalization of Lemma \ref{genmerge} to higher dimensions. Lemma \ref{genmerge} is essential because it asserts that whether merging two values in a discrete distribution increases, decreases, or preserves the functionals $L_N$ is determined by only the sum of the two value to be merged. The corresponding functionals for the higher dimensional problem are more complicated and do not have the same convenient property.\\

Still, a first step would be to verify Conjecture \ref{mainconj}(i), that the set of extreme points of $\conv(V_n)$ is precisely $V_n$. Observe that
\begin{equation}
\boldsymbol{v}_{n,m}=\Big(\frac{S(n,k)(m)_{k\downarrow}}{m^n}:1\leq k\leq n\Big)
\end{equation}
where $S(n,k)$ denotes a Stirling number of the second kind, and $(m)_{k\downarrow}$ is the falling factorial
\begin{equation}
(m)_{k\downarrow}:=m(m-1)\cdots(m-k+1)=\frac{m!}{(m-k)!}.
\end{equation}
This claim that the $\boldsymbol{v}_{n,m}$, $m=1,2,\ldots,\infty$ are the extreme points of an $(n-1)$-dimensional convex body in $\R^n$ does not seem to be recorded anywhere in the vast literature on Stirling numbers. We have verified computationally using SciPy's spatial module that $\{\boldsymbol{v}_{n,m}:1\leq m\leq 30\}$ is the set of the extreme points of its own convex hull for $n\leq 7$, but numerical precision becomes an issue for larger values of $m$ and $n$.

\section{Finite exchangeable sequences} \label{finitesection}
In this section, we consider the distribution of $K_n$ for a \textit{finite exchangeable sequence} $(X_1,\ldots,X_m)$ with $m\geq n$. Note the deviation from the original problem: the first $m$ terms of an infinite exchangeable sequence always form a finite exchangeable sequence, but a finite exchangeable sequence need not have an embedding into an infinite one, nor one with more terms. Therefore, the set of possible laws of $K_n$ for finite exchangeable sequences $(X_1,\ldots,X_m)$ form decreasing nested subsets for $m\geq n$, all of which contain that for infinite exchangeable sequences. To analyze this problem, we shift to the framework of \textit{exchangeable random partitions}, for which we provide some background below. \\\\
A \textit{partition} of $[m]:=\{1,\ldots,m\}$ is an unordered collection of disjoint non-empty subsets $\{A_i\}$ of $[m]$ with $\bigcup_i A_i=[m]$. The $A_i$ are called the \textit{clusters} of the partition. The \textit{restriction} of a partition $\{A_i\}$ of $[m]$ to $[n]$ where $n<m$ is the partition of $[n]$ whose clusters are the nonempty members of $\{A_i\cap[n]\}$.\\\\
Any infinite sequence of random variables $(X_1,X_2,\ldots)$ induces a random partition of $\N$ according to the relation $i\sim j$ if and only if $X_i=X_j$. More precisely, a random partition $\Pi$ of $\N$ is a sequence $(\Pi_m)$ where for each $m$, $\Pi_m$ is a random partition of $[m]$, and for $n<m$, the restriction of $\Pi_m$ to $[n]$ is $\Pi_n$. For the random partition $\Pi$ of $\N$ induced by a sequence $(X_1,X_2,\ldots)$, the clusters of $\Pi_m$ are the indices associated to each distinct value among $\{X_1,\ldots,X_m\}$. For example, if $$(X_1(\omega),X_2(\omega),\ldots)=(7,6,7,8,8,7\ldots),$$ then $$\Pi_1(\omega)=\{\{1\}\},\qquad \Pi_2(\omega)=\{\{1\},\{2\}\},\qquad \Pi_3(\omega)=\{\{1,3\},\{2\}\},$$ $$\Pi_4(\omega)=\{\{1,3\},\{2\},\{4\}\}\qquad \Pi_5(\omega)=\{\{1,3\},\{2\},\{4,5\}\}\qquad \Pi_6(\omega)=\{\{1,3,6\},\{2\},\{4,5\}\}.$$\\
Observe that $K_n$ as previously defined for a sequence $(X_1,X_2,\ldots)$ counts the number of clusters of $\Pi_n$ for the associated partition $\Pi$.
When $(X_1,X_2,\ldots)$ is exchangeable, it induces an \textit{exchangeable random partition} $\Pi$ of $\N$, meaning that for each $m$, the distribution of $\Pi_m$ is invariant under any deterministic permutation of $[m]$. In this scenario, associated to $\Pi$ is a function $p$ defined for all finite sequences of positive integers such that for any $m$ and any partition $\{A_1,\ldots,A_k\}$ of $[m]$, 
\begin{equation}
\PP(\Pi_m=\{A_1,\ldots,A_k\})=p(\lvert A_1\rvert,\ldots,\lvert A_k\rvert).
\end{equation}
Here $p$ is called the \textit{exchangeable partition probability function (EPPF)} associated to $\Pi$. A consequence of exchangeability is that the EPPF is a symmetric function of its arguments. The probability mass function for $K_n$ can therefore be expressed in terms of the EPPF as
\begin{equation} \label{cpcombo}
\PP(K_n=k)=\sum_{\substack{n_1+\ldots+n_k=n\\ n_1\geq\ldots\geq n_k\geq 1}}C(n_1,\ldots,n_k)p(n_1,\ldots,n_k)
\end{equation} 
where
\begin{equation} \label{cformula}
C(n_1,\ldots,n_k):=\frac{n!}{\prod_{j=1}^n(j!)^{s_j}s_j!},\qquad s_j=s_j(n_1,\ldots,n_k):=\#\{i:n_i=j\}
\end{equation}
counts the number of partitions of $[n]$ whose cluster sizes in descending order are given by $n_1,\ldots,n_k$. Furthermore, the EPPF $p$ must satisfy the following \textit{consistency} relation:
\begin{equation} \label{eppfrec}
p(n_1,\ldots,n_k)=p(n_1,\ldots,n_k,1)+\sum_{i=1}^k p(n_1,\ldots,\ n_i+1\ ,\ldots,n_k).
\end{equation}

Reposed in this alternate framework, the goal of this section is to understand the possible distributions of $K_n=K_n(\Pi_m)$ for an exchangeable random partition $\Pi_m$ of $[m]$ for $m\geq n$, meaning the number of clusters of the restriction $\Pi_{m\downarrow n}$ of $\Pi_m$ to $[n]$. A consequence of the exchangeability of $\Pi_m$ is that $\Pi_{m\downarrow n}$ is an exchangeable random partition of $[n]$, whose EPPF is the unique extension of the EPPF for $\Pi_m$ to positive integer compositions of $n$ according to the consistency relations \eqref{eppfrec}. Note that for $m=n$, $K_n(\Pi_n)$ can have any general probability distribution on $[n]$: for example, given such a probability distribution $(a_1,\ldots,a_n)$, define an EPPF according to
\begin{equation}
p(n-k+1,\underbrace{1,\ldots,1}_{k-1\text{ singletons}})=\frac{a_k}{\binom{n}{k-1}},\qquad k=1,\ldots,n
\end{equation}
where the rest of the values are either 0 or specified by symmetry. By construction, $p$ corresponds to an exchangeable random partition of $[n]$ such that $\PP(K_n=k)=a_n$ for $1\leq k\leq n$. However, for $m>n$, the consistency relations \eqref{eppfrec} must be satisfied, so it is not immediately clear given $n$ and $m>n$ what restrictions there are on the distribution of $K_n$, if any.
\begin{proposition} \label{sharpnnm1}
For $n\geq 3$, we have the sharp bound
\begin{equation}
\PP(K_n(\Pi_{n+1})=n-1)\leq\frac{\max\{4,n-1\}}{n+1}
\end{equation}
\end{proposition}
\begin{proof}
We have
\begin{equation} \label{nnm1}
\PP(K_n=n-1)=\binom{n}{2}p(2,1^{n-2})=\binom{n}{2}\big[p(3,1^{n-2})+(n-2)p(2,2,1^{n-3})+p(2,1^{n-1})\big]
\end{equation}
We consider the appearance of each of the three terms $p(3,1^{n-2})$, $p(2,2,1^{n-3})$, and $p(2,1^{n-1})$ in the expansion \eqref{eppfrec} of $p(n_1,\ldots,n_k)$ for $(n_1,\ldots,n_k)$ with $\sum_{i=1}^k n_i=n$ and $n_1\geq\ldots\geq n_k\geq 1$.
\begin{itemize}
\item $p(3,1^{n-2})$ appears in the expansion of only $p(2,1^{n-2})$ with coefficient $1$ and $p(3,1^{n-3})$ with coefficient 1. $p(3,1^{n-3})$ appears in the expansion of $\PP(K_n=n-2)$ according to \eqref{cpcombo} with coefficient $C(3,1^{n-3})=\binom{n}{3}$.
\item $p(2,2,1^{n-3})$ appears in the expansion of only $p(2,1^{n-2})$ with coefficient $n-2$ and $p(2,2,1^{n-4})$ with coefficient 1. $p(2,2,1^{n-4})$ appears in the expansion of $\PP(K_n=n-2)$ according to $\eqref{cpcombo}$ with coefficient $C(2,2,1^{n-4})=3\binom{n}{4}$.
\item $p(2,1^{n-1})$ appears in the expansion of only $p(2,1^{n-2})$ with coefficient $1$ and $p(1^{n})$ with coefficient $n$. $p(1^n)$ appears in the expansion of $\PP(K_n=n)$ with coefficient $C(1^n)=1$.
\end{itemize}
Hence the problem reduces to maximizing \eqref{nnm1} subject to the linear constraints
\begin{equation}
\bigg[\binom{n}{2}+\binom{n}{3}\bigg]p(3,1^{n-2})+\bigg[\binom{n}{2}(n-2)+3\binom{n}{4}\bigg]p(2,2,1^{n-3})+\bigg[\binom{n}{2}+n\bigg]p(2,1^{n-1})\leq 1.
\end{equation}
The maximum value of \eqref{nnm1} is evidently equal to
\begin{equation}
\max\bigg\{\frac{\binom{n}{2}}{\binom{n}{2}+\binom{n}{3}},\frac{\binom{n}{2}(n-2)}{\binom{n}{2}(n-2)+3\binom{n}{4}},\frac{\binom{n}{2}}{\binom{n}{2}+n}\bigg\},
\end{equation}
and simplifying each of the three expressions yields
\begin{equation}
\max\Big\{\frac{3}{n+1},\frac{4}{n+1},\frac{n-1}{n+1}\Big\}=\frac{\max\{4,n-1\}}{n+1}.
\end{equation} 
\end{proof}
It follows from Proposition \ref{sharpnnm1} that for $n=3$, there are no restrictions on the distribution of $K_3(\Pi_4)$ on $\{1,2,3\}$. The corresponding claim cannot be made for $n\geq 4$, as $\PP(K_4(\Pi_5)=3)\leq\frac{4}{5}$ and $\PP(K_n(\Pi_{n+1})=n-1)\leq\frac{n-1}{n+1}$ for $n\geq 5$.\\

The remainder of the section will focus on $K_3(\Pi_n)$ for $n\geq 3$. Intuitively, as $n\rightarrow\infty$, the set of probability distributions of $K_3(\Pi_n)$ should tend to the corresponding set for $K_3(\Pi)$ for exchangeable random partitions $\Pi$ of $\N$, which was explicitly characterized in Section 2. We proceed by fixing $n\geq 3$, and as before, consider the parameterization $q_1=\PP(K_3(\Pi_n)=1)$ and $q_3=\PP(K_3(\Pi_n)=3)$. By repeated application of \eqref{eppfrec}, $q_1$ and $q_3$ may be written in terms of the EPPF as
\begin{equation}
q_1=p(3)=\sum_{\substack{1\leq k\leq n\\n_1+\ldots+n_k=n\\ n_1\geq\ldots\geq n_k\geq 1}}A(n_1,\ldots,n_k)p(n_1,\ldots,n_k)
\end{equation}
and
\begin{equation}
q_3=p(1,1,1)=\sum_{\substack{1\leq k\leq n\\ n_1+\ldots+n_k=n\\ n_1\geq\ldots\geq n_k\geq 1}}B(n_1,\ldots,n_k)p(n_1,\ldots,n_k)
\end{equation}
for uniquely defined nonnegative integer coefficients $A(n_1,\ldots,n_k)$ and $B(n_1,\ldots,n_k)$.The problem is to describe the set of points $(q_1,q_3)$ arising in this manner subject to 
\begin{equation}
\sum_{\substack{1\leq k\leq n\\ n_1+\ldots+n_k=n\\ n_1\geq\ldots\geq n_k\geq 1}}C(n_1,\ldots,n_k)p(n_1,\ldots,n_k)=1
\end{equation}
where $C(n_1,\ldots,n_k)$ is as defined in \eqref{cformula}. Observe that, in vector notation,
\begin{align}
(q_1,q_3)&=\Big(\sum A(n_1,\ldots,n_k)p(n_1,\ldots,n_k), \sum B(n_1,\ldots,n_k)p(n_1,\ldots,n_k)\Big)\\
&=\sum C(n_1,\ldots,n_k)p(n_1,\ldots,n_k)\Big(\tfrac{A(n_1,\ldots,n_k)}{C(n_1,\ldots,n_k)},\tfrac{B(n_1,\ldots,n_k)}{C(n_1,\ldots,n_k)}\Big)
\end{align}
This shows that any $(q_1,q_3)$ is a convex combination of points of the form $\big(\frac{A(\boldsymbol{n})}{C(\boldsymbol{n})},\frac{B(\boldsymbol{n})}{C(\boldsymbol{n})}\big)$, and thus the set of probability distributions of $K_3(\Pi_n)$ over all exchangeable random partitions $\Pi_n$ of $[n]$, expressed in the parametrization $(q_1,q_3)$, is the convex hull of the finite set of points
\begin{equation}
S_n:=\Big\{\Big(\tfrac{A(n_1,\ldots,n_k)}{C(n_1,\ldots,n_k)},\tfrac{B(n_1,\ldots,n_k)}{C(n_1,\ldots,n_k)}\Big):1\leq k\leq n, n_1+\ldots+n_k=1,n_1\geq\ldots\geq n_k\geq 1\Big\}.
\end{equation}
\vspace{-1cm}
\begin{figure}[h!]
\centering
  \centering
  \includegraphics[width=0.7\linewidth]{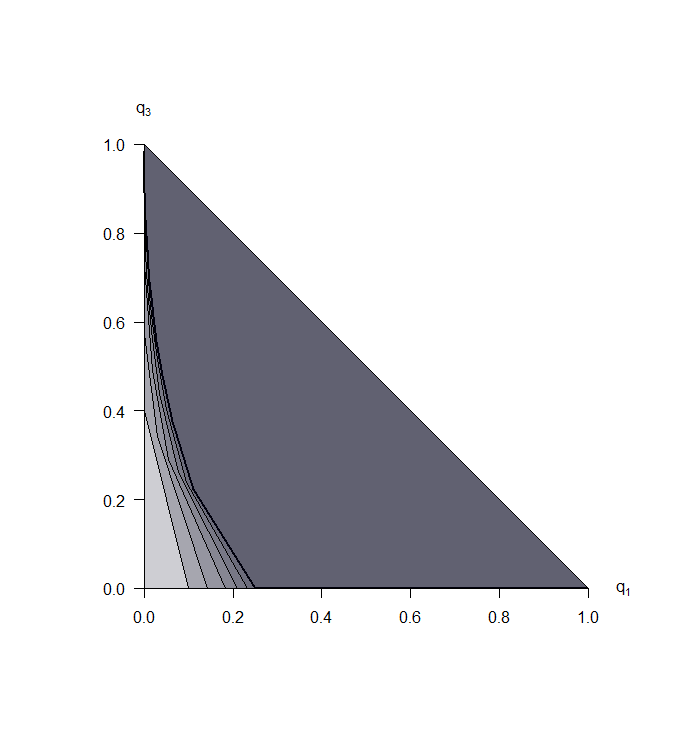}
  \vspace{-1.7cm}
  \captionsetup{width=0.5\linewidth}
  \captionof{figure}{The nested regions are the possible probability distributions of $K_3(\Pi_n)$ for $\Pi_n$ an exchangeable random partition of $[n]$ for $n=4,5,7,12,19,41$, which tend to the region corresponding to $K_3$ for infinite exchangeable sequences, as described in Theorem \ref{mainthm} and shown in Figure \ref{segments}.} 
  \label{457}
\end{figure}\\\\
\newpage
Listed below is the sequence $(s_n)$ for the number of extreme points of the convex hull of $S_n$, $n\geq 3$:

\begin{table}[h]
\begin{center}
\begin{tabular}{|c||c|c|c|c|c|c|c|c|c|c|c|c|c|c|c|c|c|c|c|c|c|c|c|c|}
\hline
$n$ & 3 & 4 & 5 & 6 & 7 & 8 & 9 & 10 & 11 & 12 & 13 & 14 & 15 & 16 & 17 & 18 & 19 & 20 & 21 & 22 & 23
\\ \hline
$s_n$ & 3 & 3 & 4 & 4 & 5 & 5 & 6 & 6 & 7 & 6 & 8 & 7 & 8 & 8 & 9 & 8 & 10 & 9 & 10 & 10 & 11\\
\hline
\end{tabular}
\newline\newline\newline
\begin{tabular}{|c||c|c|c|c|c|c|c|c|c|c|c|c|c|c|c|c|c|c|c|c|c|c|c|c|}
\hline
$n$ & 24 & 25 & 26 & 27 & 28 & 29 & 30 & 31 & 32 & 33 & 34 & 35 & 36 & 37 & 38 & 39 & 40 & 41
\\ \hline
$s_n$ & 9 & 12 & 11 & 11 & 11 & 13 & 11 & 13 & 12 & 13 & 13 & 14 & 12 & 15 & 14 & 14 & 13 & 16 
\\ \hline
\end{tabular}
\end{center}
\end{table}

\section{The two-parameter family} \label{twoparamsection}
It was shown in \cite{MR1337249} that any pair of real parameters $(\alpha,\theta)$ satisfying either of the conditions
\begin{align}
&\text{(i)}\ \ 0\leq\alpha<1 \text{ and } \theta>-\alpha \text{;  or}\\
&\text{(ii)}\ \ \alpha<0 \text{ and } \theta=-m\alpha \text{ for some } m\in\N
\end{align}
correspond to an exchangeable random partition $\Pi_{\alpha,\theta}=(\Pi_n)$ of $\N$ according to the following sequential construction known as the Chinese restaurant process: for each $n\in\N$, conditionally given $\Pi_n=\{C_1,\ldots,C_k\}$, $\Pi_{n+1}$ is formed by having $n+1$
\begin{equation}
\begin{split}
&\text{attach to cluster }C_i\text{ with probability }\frac{\lvert C_i\rvert-\alpha}{n+\theta},\ \ 1\leq i\leq k\ ;\\
&\text{form a new cluster with probability }\frac{\theta+k\alpha}{n+\theta}.
\end{split}
\end{equation}
The corresponding EPPF is given by
\begin{equation} \label{ateppf}
p_{\alpha,\theta}(n_1,\ldots,n_k)=\frac{\prod_{i=0}^{k-1}(\theta+i\alpha)\prod_{j=1}^k(1-\alpha)_{n_j-1}}{(\theta)_n}
\end{equation} 
where $n=n_1+\ldots+n_k$ and
\begin{equation}
(x)_m:=x(x+1)\cdots(x+m-1)=\frac{\Gamma(x+m)}{\Gamma(x)}.
\end{equation}
Let $\PP_{\alpha,\theta}$ denote the law of $\Pi_{\alpha,\theta}$.
\begin{figure}[h!]
\vspace{-1cm}
\centering
  \centering
  \includegraphics[width=0.7\linewidth]{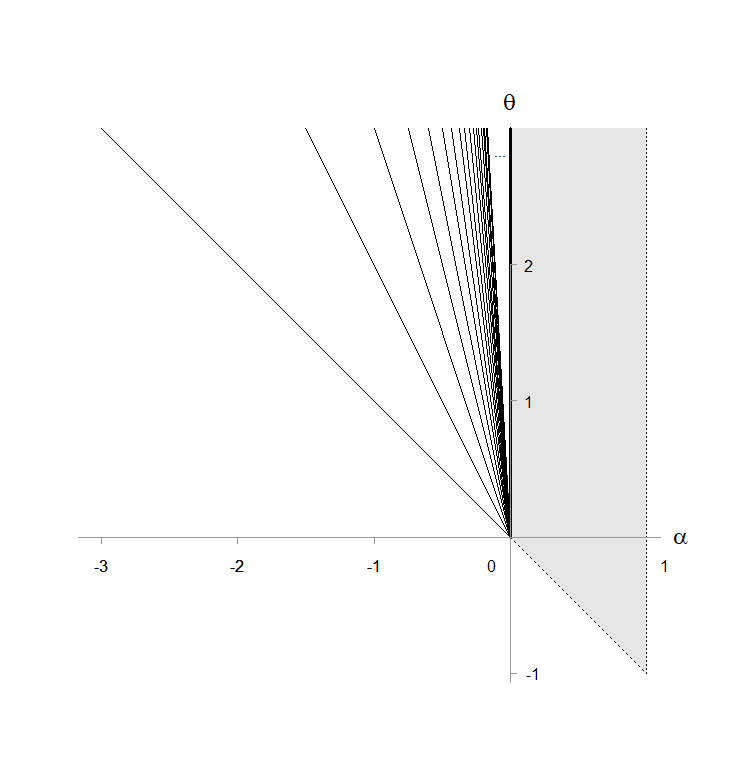}
  \captionsetup{width=0.8\linewidth}
  \vspace{-1.2cm}
  \captionof{figure}{The $(\alpha,\theta)$ parameter space.} 
  \label{atpspace}
\end{figure}
The distribution of $K_3$ for $\Pi_{\alpha,\theta}$ is given by
\begin{align}
q_1(\alpha,\theta)&=\frac{(1-\alpha)(2-\alpha)}{(1+\theta)(2+\theta)} \label{q1at}\\
q_2(\alpha,\theta)&=\frac{3(1-\alpha)(\theta+\alpha)}{(1+\theta)(2+\theta)} \label{q2at}\\
q_3(\alpha,\theta)&=\frac{(\theta+\alpha)(\theta+2\alpha)}{(1+\theta)(2+\theta)} \label{q3at}
\end{align} 
where 
\begin{equation}
q_i(\alpha,\theta):=\PP_{\alpha,\theta}(K_3=i).
\end{equation}

For $m>0$, let
\begin{equation}
A_m:=\big\{(m+m\theta,\theta):-\tfrac{m}{m+1}<\theta<\tfrac{1-m}{m}\big\}\subseteq\{(\alpha,\theta):0\leq\alpha<1,\theta>-\alpha\}
\end{equation}
and let $A_0:=\{(0,\theta):\theta>0\}$, the parameter subspace corresponding to the well-known one-parameter Ewens sampling formula \cite{MR325177}.
The line segments and one ray $\{A_m\}_{m\geq 0}$ with inverse slope $m$ in the $(\alpha,\theta)$ plane, each of which would pass through the point $(\alpha,\theta)=(0,-1)$ if extended, partition the parameter subspace $\{(\alpha,\theta):0\leq \alpha<1,\theta>-\alpha\}$. Hence the distribution of $K_3$ can be reparametrized in $m$ and $\theta$ as
\begin{align}
q_1^{(m)}(\theta)&=\frac{(1-m-m\theta)(2-m-m\theta)}{(1+\theta)(2+\theta)} \label{q1m}\\ 
q_2^{(m)}(\theta)&=\frac{3(1-m-m\theta)[m+(m+1)\theta]}{(1+\theta)(2+\theta)} \label{q2m}\\
q_3^{(m)}(\theta)&=\frac{[m+(m+1)\theta][2m+(2m+1)\theta]}{(1+\theta)(2+\theta)} \label{q3m}
\end{align} 
It can be checked by calculus that for each fixed $m>0$, 
\begin{itemize}
\item the function $q_1^{(m)}(\theta)$ is strictly decreasing for $\theta\in(-\tfrac{m}{m+1},\tfrac{1-m}{m})$ with $\lim_{\theta\rightarrow -\frac{m}{m+1}}q_1^{(m)}(\theta)=1$ and $\lim_{\theta\rightarrow\frac{1-m}{m}}q_1^{(m)}(\theta)=0$
\item the function $q_3^{(m)}(\theta)$ is strictly increasing for $\theta\in(-\tfrac{m}{m+1},\tfrac{1-m}{m})$ with $\lim_{\theta\rightarrow -\frac{m}{m+1}}q_3^{(m)}(\theta)=0$ and $\lim_{\theta\rightarrow\frac{1-m}{m}}q_3^{(m)}(\theta)=1$
\item the function $q_2^{(m)}(\theta)$ is strictly increasing on $\big(-\frac{m}{m+1},\tau(m)\big]$ and strictly decreasing on $\big[\tau(m),\frac{1-m}{m}\big)$, with a unique maximum value of $9-6\big(\sqrt{(m+1)(m+2)}-m\big)$ at $\theta=\tau(m):=\frac{-m^2-3m+\sqrt{(m+1)(m+2)}}{1+3m+m^2}$, which is also the unique value of $\theta$ in the domain at which $q_1^{(m)}(\theta)=q_3^{(m)}(\theta)$.
\end{itemize} 
The properties above also hold for $m=0$ after slight modification by replacing each instance of $\frac{1-m}{m}$ with $\lim_{m\rightarrow 0^+}\frac{1-m}{m}=\infty$, and this remark also applies to subsequent discussion.
\begin{figure}[h!]
\centering
  \centering
  \includegraphics[width=0.9\linewidth]{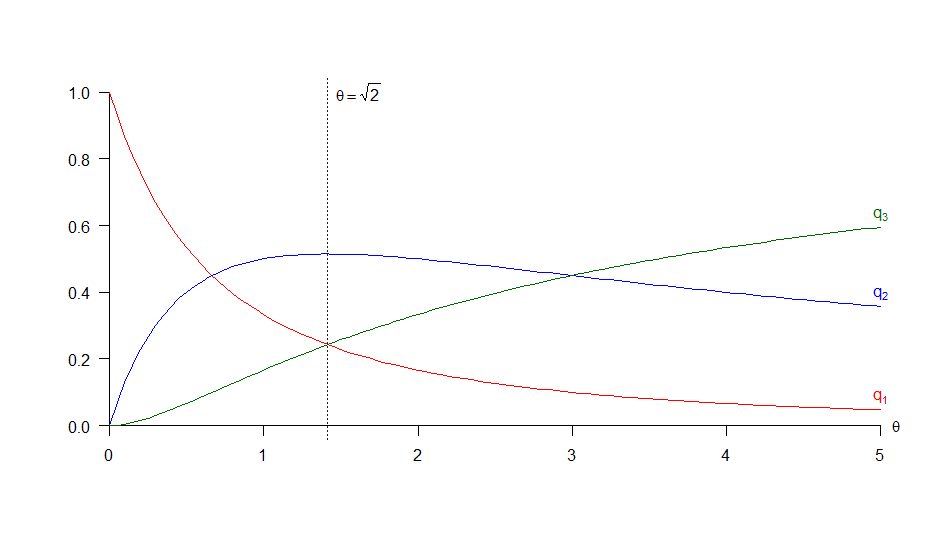}
  \captionsetup{width=0.6\linewidth}
  \captionof{figure}{Graphs of $q_i^{(m)}(\theta)$ for $m=0$ and $\theta\in[0,5]$. Observe that $q_1$ and $q_3$ intersect at the same value of $\theta$ as where $q_2$ attains its maximum value. The corresponding graphs for every $m>0$ also share this property.} 
  \label{ewensk3}
\end{figure}
\vspace{1cm}\\
\textbf{Duality.} The last observation implies that for $m\geq 0$ and any real number $p$ such that $0<p<9-6(\sqrt{(m+1)(m+2)}-m)$, there are exactly two values $\theta_{\pm}^{(m)}(p)$ with 
\begin{equation} \label{dualthetas}
-\frac{m}{m+1}<\theta_-^{(m)}(p)<\tau(m)<\theta_+^{(m)}(p)<\frac{1-m}{m}.
\end{equation}
satisfying
\begin{equation}
q_2^{(m)}(\theta_-^{(m)}(p))=q_2^{(m)}(\theta_+^{(m)}(p)).
\end{equation}
For $p=9-6(\sqrt{(m+1)(m+2)}-2)$, define $\theta_-^{(m)}(p)=\theta_+^{(m)}(p)=\varphi(m)$. As $\theta_{\pm}^{(m)}(p)$ are defined as the solutions to the equation
\begin{equation}
\frac{3(1-m-m\theta)[m+(m+1)\theta]}{(1+\theta)(2+\theta)}=p
\end{equation}
or equivalently the quadratic equation
\begin{equation} \label{defdualq}
p(1+\theta)(2+\theta)-3(1-m-m\theta)[m+(m+1)\theta]=0,
\end{equation}
we have the polynomial identity
\begin{equation}
(\theta-\theta_+^{(m)}(p))(\theta-\theta_-^{(m)}(p))=\theta^2+\frac{3p-3+6m^2}{p+3m+3m^2}\theta+\frac{2p-3m+3m^2}{p+3m+3m^2}
\end{equation}
after rearranging \eqref{defdualq}. It follows that
\begin{equation} \label{dualmult}
\theta_+^{(m)}\theta_-^{(m)}=\frac{2p-3m+3m^2}{p+3m+3m^2}.
\end{equation}
For $-\frac{m}{m+1}<\theta<\frac{1-m}{m}$, define the $m$-\textit{dual} $\theta_*^{(m)}$ of $\theta$ according to \eqref{dualthetas}. Rearranging \eqref{dualmult} and simplifying gives the explicit formula
\begin{equation} \label{dualform}
\theta_*^{(m)}=\frac{2-m(3+m)(1+\theta)}{\theta+m(3+m)(1+\theta)}.
\end{equation}
\begin{figure}[h!]
\centering
  \includegraphics[width=1\linewidth]{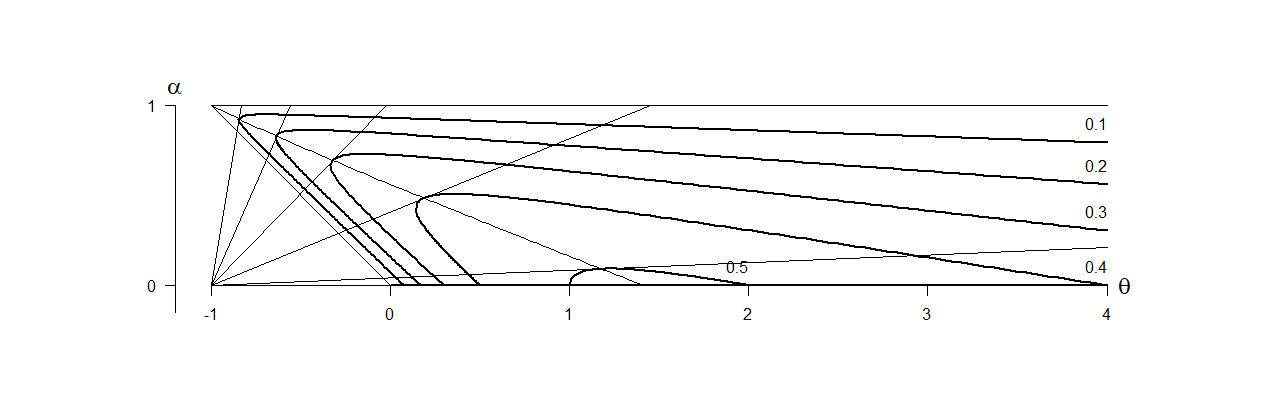}
  \captionsetup{width=0.8\linewidth}
  \vspace{-1.2cm}
  \captionof{figure}{Contour plot of $q_2(\alpha,\theta)$. The level curves for $q_2(\alpha,\theta)\in\{0.1,0.2,0.3,0.4,0.5\}$ are shown, along with their tangent lines where they meet the curve $q_1(\alpha,\theta)=q_3(\alpha,\theta)$. Observe that each tangent line passes through the point $(\alpha,\theta)=(0,-1)$. Note that here $\alpha$ is plotted on the vertical axis, for convenience of display.}
  \label{q2contour}
\end{figure}
\begin{theorem} \label{mduality}
For $m\geq 0$ and $-\frac{m}{m+1}<\theta<\frac{1-m}{m}$, we have 
\begin{equation}
q_1^{(m)}(\theta_*^{(m)})=q_3^{(m)}(\theta)\qquad\text{and}\qquad q_3^{(m)}(\theta_*^{(m)})=q_1^{(m)}(\theta).
\end{equation}
\end{theorem}
\begin{proof}
It suffices to verify the first of the two identities since \eqref{dualform} is constructed as an involution. Let $D(m,\theta)$ be the denominator in \eqref{dualform}. Substituting and simplifying yields
\begin{align}
1+\theta^{(m)}_*&=\frac{2+\theta}{D(m,\theta)}\ ;\\
2+\theta^{(m)}_*&=\frac{(1+\theta)(1+m)(2+m)}{D(m,\theta)}\ ;\\
1-m-m\theta^{(m)}_*&=\frac{(1+m)[m+(m+1)\theta]}{D(m,\theta)}\ ;\\
2-m-m\theta^{(m)}_*&=\frac{(2+m)[2m+(2m+1)\theta]}{D(m,\theta)}.
\end{align}
Hence we have
\begin{equation}
q_1^{(m)}(\theta_*^{(m)})=\frac{(1-m-m\theta_*^{(m)})(2-m-m\theta_*^{(m)})}{(1+\theta_*^{(m)})(2+\theta_*^{(m)})}=\frac{[m+(m+1)\theta][2m+(2m+1)\theta]}{(1+\theta)(2+\theta)}=q_3^{(m)}(\theta)
\end{equation}
as desired.\\
\end{proof}
\textbf{Symmetry.} A consequence of Theorem \ref{mduality} is a surprising symmetry in the set of laws of $K_3$ arising from the two-parameter model. To make this observation explicit, for any $m\geq 0$ we solve for $q_3=q_3^{(m)}$ in terms of $q_1=q_1^{(m)}$ as defined in \eqref{q1m} and \eqref{q3m} to obtain the formula 
\begin{equation} \label{q1q3m}
q_3=\varphi_m(q_1):=1+\frac{3}{4}m+\frac{5}{4}q_1-\frac{3}{4}\sqrt{m^2+6q_1m+q_1(8+q_1)}.
\end{equation}
Rearranging to eliminate the radical yields the relation
\begin{equation}
(4+3m)(q_1+q_3)+5q_1q_3-2(q_1^2+q_3^2)-2-3m=0
\end{equation}
which verifies the symmetry. For $m=0$ the identity reduces to
\begin{equation} \label{hqq}
h(q_1,q_3):=4(q_1+q_3)+5q_1q_3-2(q_1^2+q_3^2)-2=0.
\end{equation}
\begin{theorem} \label{bijthm}
The mapping $(\alpha,\theta)\mapsto(q_1,q_3)$ defined by \eqref{q1at} and \eqref{q3at} is a bijection between the regions 
\begin{equation}
\{(\alpha,\theta):0\leq\alpha<1,\ \theta>-\alpha\}\qquad\text{and}\qquad\{(q_1,q_3):h(q_1,q_3)\geq 0,\ q_1+q_3<1\}
\end{equation}
where $h(q_1,q_3)$ is defined as in \eqref{hqq}.
\end{theorem}
\begin{proof}
Consider $\varphi(m,q_1):=\varphi_m(q_1)$ as in \eqref{q1q3m}. To show the desired bijection, it suffices to show that for every fixed $0<q_1<1$ that (i) $\varphi(m,q_1)$ is increasing in $m$, and (ii) $\lim_{m\rightarrow\infty}\varphi(m,q_1)=1-q_1$.\\\\
(i) 
\begin{equation}
\frac{\partial}{\partial m}\varphi(m,q_1)=\frac{3}{4}(1-\frac{2m+6q_1}{2\sqrt{m^2+6q_1m+q_1(8+q_1)}})>\frac{3}{4}(1-\frac{2m+6q_1}{2\sqrt{m^2+6q_1m+9q_1^2}})=0
\end{equation}
(ii) 
\begin{align}
\lim_{m\rightarrow\infty}\varphi(m,q_1)&=\lim_{m\rightarrow\infty}1+\frac{5}{4}q_1+\frac{3}{4}\bigg(\frac{m^2-(m^2+6q_1m+q_1(8+q_1))}{m+\sqrt{m^2+6q_1m+q_1(8+q_1)}}\bigg)\\
&=\lim_{m\rightarrow\infty}1+\frac{5}{4}q_1+\frac{3}{4}\bigg(\frac{-6q_1-\frac{q_1(8+q_1)}{m}}{1+\sqrt{1+\frac{6q_1}{m}+\frac{q_1(8+q_1)}{m^2}}}\bigg)\\
&=1-q_1
\end{align}
\end{proof}
\begin{figure}[h!]
\vspace{-0.5cm}
\centering
\begin{minipage}{0.35\textwidth}
  \centering
  \includegraphics[width=0.545\linewidth]{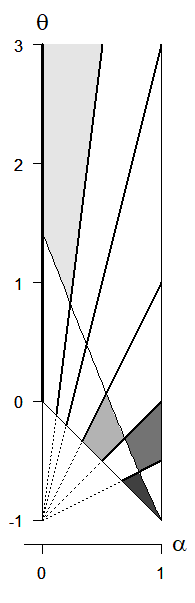}
  \label{atcuts}
\end{minipage}%
\begin{minipage}{0.7\textwidth}
  \centering
  \includegraphics[width=1\linewidth]{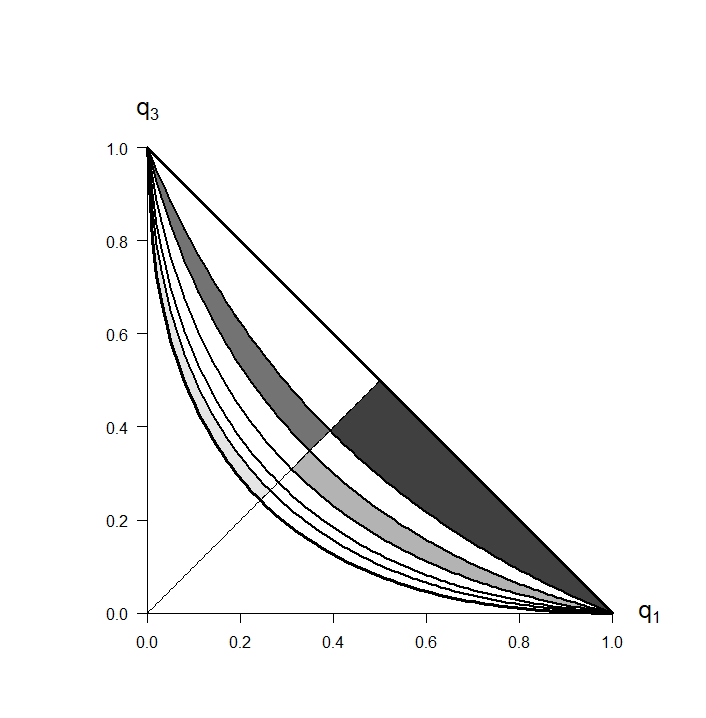}
  \label{q1q3cuts}
\end{minipage}
\captionsetup{width=1\linewidth}
\vspace{-1cm}
  \captionof{figure}{The bijection of Theorem \ref{bijthm}. The regions colored in different shades of gray reveal the geometry of the bijection.}
\end{figure}

\textbf{Explicit inverse.} Define the ratios
\begin{equation}
r(\alpha,\theta):=\frac{q_1(\alpha,\theta)}{q_2(\alpha,\theta)}=\frac{2-\alpha}{3(\theta+\alpha)},\qquad s(\alpha,\theta):=\frac{q_2(\alpha,\theta)}{q_3(\alpha,\theta)}=\frac{3(1-\alpha)}{(\theta+2\alpha)}
\end{equation}
These ratios uniquely define the law of $K_3$ for the corresponding $(\alpha,\theta)$. The map $(\theta,\alpha)\mapsto(r,s)$ can be explicitly inverted as 
\begin{equation}
\alpha(r,s)=\frac{9r-2s}{9r-s+3rs},\qquad\theta(r,s)=\frac{3-9r+4s}{9r-s+3rs}
\end{equation}
Expressed in terms of $q_1$ and $q_3$, this gives the inversion formulas
\begin{equation}
\alpha(q_1,q_3)=\frac{4q_1+4q_3+5q_1q_3-2q_1^2-2q_3^2-2}{5q_1+2q_3+4q_1q_3-4q_1^2-q_3^2-1},\qquad
\theta(q_1,q_3)=-\frac{8q_1+5q_3+4q_1q_3-4q_1^2-q_3^2-4}{5q_1+2q_3+4q_1q_3-4q_1^2-q_3^2-1}
\end{equation}
Note that the numerator in the formula for $\alpha(q_1,q_3)$ is equal to $h(q_1,q_3)$ as defined in \eqref{hqq}. It is easy to verify that these formulas give an algebraic inverse. Observe that the denominator which is the same in both formulas is nonvanishing on the region $\{(q_1,q_3):h(q_1,q_3)\geq0,\ q_1+q_3<1\}$, since 
\begin{align}
2(5q_1+2q_3+4q_1q_3-4q_1^2-q_3^2-1)&=h(q_1,q_3)+6q_1-6q_1^2+3q_1q_3>0.
\end{align}
\begin{corollary}
For any parameters $(\alpha,\theta)$ with $0\leq\alpha<1$ and $\theta>-\alpha$, there exists a unique pair $(\alpha_*,\theta_*)$ with $0\leq\alpha_*<1$ and $\theta_*>-\alpha_*$ such that 
\begin{equation}
q_{2\pm 1}(\alpha,\theta)=q_{2\mp 1}(\alpha_*,\theta_*).
\end{equation}
\end{corollary}
Explicit formulas for $\alpha_*$ and $\theta_*$ in terms of $\alpha$ and $\theta$ can be computed as
\begin{align}
\alpha^*&=\frac{(2-3\alpha)(1+\theta)-\alpha^2}{(\theta+3\alpha)(1+\theta)+\alpha^2}\\
\theta^*&=\frac{\alpha(2+\theta)}{(\theta+3\alpha)(1+\theta)+\alpha^2}.
\end{align} 
\textbf{Exceptional parameters.} $\alpha<0$, $\theta=-m\alpha$ for some $m\in\N$\\\\
It is well-known that in this case, the exchangeable random partition $(\Pi_n)$ of $\N$ generated according to the Chinese restaurant construction is distributed as if by sampling from a symmetric Dirichlet distribution with $m$ parameters equal to $-\alpha$ \cite{MR2245368}. Hence for fixed $m\in\N$, as $\alpha\downarrow -\infty$ the exchangeable random partition of $\N$ corresponding to the parameter pair $(\alpha,\theta)=(\alpha,-m\alpha)$ converges in distribution to that obtained by sampling from the discrete uniform distribution on $m$ elements. For $K_3$, the $(\alpha,\theta)$ to $(q_1,q_3)$ correspondence can be seen in Figure \ref{dirichlet}.\\
\vspace{-0.6cm}
\begin{figure}[h!]
\centering
  \includegraphics[width=0.7\linewidth]{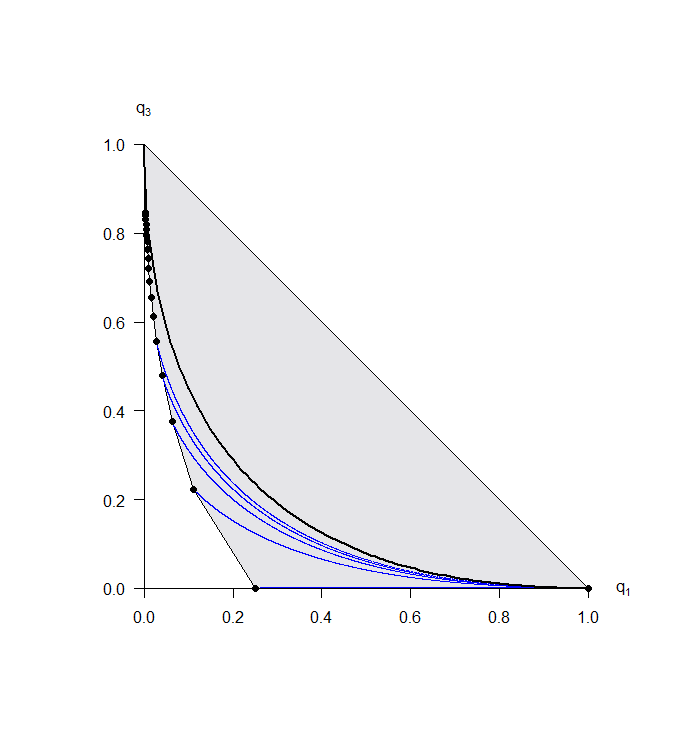}
  \captionsetup{width=0.6\linewidth}
  \vspace{-1.7cm}
  \captionof{figure}{The blue curves correspond to the images of $(\alpha,\theta)=(\alpha,-m\alpha)$ for $\alpha\in(-\infty,0)$ and fixed $m$ under the $(\alpha,\theta)\mapsto(q_1,q_3)$ map, for $m=2,3,4,5,6$. The curve defined by \eqref{hqq} is included in black.}
  \label{dirichlet}
\end{figure}

\newpage

\section{Complements} \label{looseends}
In this section, we point out an interesting convexity property for the the law of $K_3$. With notation as in Section \ref{k3main}, for $\boldsymbol{p}\in\nabla_\infty$, let
\begin{equation}
\boldsymbol{Q}(\boldsymbol{p}):=\big(q_1(\boldsymbol{p}),q_3(\boldsymbol{p})\big)
\end{equation}
be the mapping from a ranked discrete distribution to its corresponding law of $K_3$ obtained by i.i.d. sampling. In Section \ref{k3main} we proved that the range of $\boldsymbol{Q}$ is a subset of the closed convex hull of the set of points $\{\boldsymbol{Q}(\boldsymbol{u}_N):N\in\N\}$. Here are some preliminary efforts to better understand the geometry of this mapping.
\begin{proposition} For any $0\leq\lambda\leq 1$ and $N\geq 1$,
\begin{equation}
\boldsymbol{Q}(\lambda\boldsymbol{u}_N+(1-\lambda)\boldsymbol{u}_{2N})=\lambda^2\boldsymbol{Q}(\boldsymbol{u}_N)+(1-\lambda^2)\boldsymbol{Q}(\boldsymbol{u}_{2N})
\end{equation}
\end{proposition}
\begin{proof} We have
\begin{equation}
\lambda\boldsymbol{u}_N+(1-\lambda)\boldsymbol{u}_{2N}=\big(\underbrace{\tfrac{1+\lambda}{2N},\ldots,\tfrac{1+\lambda}{2N}}_{N\text{ times}},\underbrace{\tfrac{1-\lambda}{2N},\ldots,\tfrac{1-\lambda}{2N}}_{N\text{ times}}\big).
\end{equation}
Hence
\begin{equation}
q_1(\lambda\boldsymbol{u}_N+(1-\lambda)\boldsymbol{u}_{2N})=N\Big(\frac{1+\lambda}{2N}\Big)^3+N\Big(\frac{1-\lambda}{2N}\Big)^3=\frac{1+3\lambda^2}{4N^2}
\end{equation}
and
\begin{align}
q_3&(\lambda\boldsymbol{u}_N+(1-\lambda)\boldsymbol{u}_{2N})\\&=\binom{N}{3}\Big(\frac{1+\lambda}{2N}\Big)^3+\binom{N}{2}N\Big(\frac{1+\lambda}{2N}\Big)^2\Big(\frac{1-\lambda}{2N}\Big)+N\binom{N}{2}\Big(\frac{1+\lambda}{2N}\Big)\Big(\frac{1-\lambda}{2N}\Big)^2+\binom{N}{3}\Big(\frac{1-\lambda}{2N}\Big)^3\\
&=\binom{N}{3}\frac{1+3\lambda^2}{4N^3}+N\binom{N}{2}\frac{1-\lambda^2}{4N^3}\\
&=\frac{N-1}{3}\Big(\frac{2N-1-3\lambda^2}{4N^2}\Big).
\end{align}
On the other side,
\begin{equation}
\lambda^2 q_1(\boldsymbol{u}_N)+(1-\lambda^2)q_1(\boldsymbol{u}_{2N})=\frac{\lambda^2}{N^2}+\frac{1-\lambda^2}{4N^2}=\frac{1+3\lambda^2}{4N^2}
\end{equation}
and
\begin{align}
\lambda^2 q_3(\boldsymbol{u}_N)+(1-\lambda^2)q_3(\boldsymbol{u}_{2N})&=\lambda^2\binom{N}{3}\frac{1}{N^3}+(1-\lambda^2)\binom{2N}{3}\frac{1}{8N^3}\\
&=\frac{N(N-1)(N-2)}{6}\cdot\frac{\lambda^2}{N^3}+\frac{2N(2N-1)(2N-2)}{6}\cdot\frac{1-\lambda^2}{8N^3}\\
&=\frac{N-1}{3}\Big(\frac{2N-1-3\lambda^2}{4N^2}\Big).
\end{align}
\end{proof}
\vspace{1cm}
\textbf{Acknowledgement.} Many thanks to my advisor Jim Pitman for suggesting this problem and providing invaluable guidance.

\newpage

\bibliography{k3}

\begin{thebibliography}{10}

\bibitem{MR137256}
R.~R. Bahadur.
\newblock On the number of distinct values in a large sample from an infinite
  discrete distribution.
\newblock {\em Proc. Nat. Inst. Sci. India Part A}, 26(supplement II):67--75,
  1960.

\bibitem{MR2253162}
Jean Bertoin.
\newblock {\em Random fragmentation and coagulation processes}, volume 102 of
  {\em Cambridge Studies in Advanced Mathematics}.
\newblock Cambridge University Press, Cambridge, 2006.

\bibitem{MR2412154}
Leonid~V. Bogachev, Alexander~V. Gnedin, and Yuri~V. Yakubovich.
\newblock On the variance of the number of occupied boxes.
\newblock {\em Adv. in Appl. Math.}, 40(4):401--432, 2008.

\bibitem{MR3458585}
Harry Crane.
\newblock The ubiquitous {E}wens sampling formula.
\newblock {\em Statist. Sci.}, 31(1):1--19, 2016.

\bibitem{MR954608}
Sudhakar Dharmadhikari and Kumar Joag-Dev.
\newblock {\em Unimodality, convexity, and applications}.
\newblock Probability and Mathematical Statistics. Academic Press, Inc.,
  Boston, MA, 1988.

\bibitem{MR577313}
P.~Diaconis and D.~Freedman.
\newblock Finite exchangeable sequences.
\newblock {\em Ann. Probab.}, 8(4):745--764, 1980.

\bibitem{MR2722836}
Rick Durrett.
\newblock {\em Probability: theory and examples}, volume~31 of {\em Cambridge
  Series in Statistical and Probabilistic Mathematics}.
\newblock Cambridge University Press, Cambridge, fourth edition, 2010.

\bibitem{MR325177}
W.~J. Ewens.
\newblock The sampling theory of selectively neutral alleles.
\newblock {\em Theoret. Population Biol.}, 3, 1972.

\bibitem{MR0228020}
William Feller.
\newblock {\em An introduction to probability theory and its applications.
  {V}ol. {I}}.
\newblock Third edition. John Wiley \& Sons, Inc., New York-London-Sydney,
  1968.

\bibitem{MR0204592}
David Gale and Hukukane Nikaid\^{o}.
\newblock The {J}acobian matrix and global univalence of mappings.
\newblock {\em Math. Ann.}, 159:81--93, 1965.

\bibitem{MR2318403}
Alexander Gnedin, Ben Hansen, and Jim Pitman.
\newblock Notes on the occupancy problem with infinitely many boxes: general
  asymptotics and power laws.
\newblock {\em Probab. Surv.}, 4:146--171, 2007.

\bibitem{MR2744243}
Alexander Gnedin, Chris Haulk, and Jim Pitman.
\newblock Characterizations of exchangeable partitions and random discrete
  distributions by deletion properties.
\newblock In {\em Probability and mathematical genetics}, volume 378 of {\em
  London Math. Soc. Lecture Note Ser.}, pages 264--298. Cambridge Univ. Press,
  Cambridge, 2010.

\bibitem{MR0216548}
Samuel Karlin.
\newblock Central limit theorems for certain infinite urn schemes.
\newblock {\em J. Math. Mech.}, 17:373--401, 1967.

\bibitem{MR2160323}
S.~Kerov.
\newblock Coherent random allocations, and the {E}wens-{P}itman formula.
\newblock {\em Zap. Nauchn. Sem. S.-Peterburg. Otdel. Mat. Inst. Steklov.
  (POMI)}, 325(Teor. Predst. Din. Sist. Komb. i Algoritm. Metody. 12):127--145,
  246, 2005.

\bibitem{khintchine1938unimodal}
A~Ya Khintchine.
\newblock On unimodal distributions.
\newblock {\em Izvestiya Nauchno-Issledovatel’skogo Instituta Matematiki i
  Mekhaniki}, 2(2):1--7, 1938.

\bibitem{MR509954}
J.~F.~C. Kingman.
\newblock The representation of partition structures.
\newblock {\em J. London Math. Soc. (2)}, 18(2):374--380, 1978.

\bibitem{MR671034}
J.~F.~C. Kingman.
\newblock The coalescent.
\newblock {\em Stochastic Process. Appl.}, 13(3):235--248, 1982.

\bibitem{MR2596654}
L.~A. Petrov.
\newblock A two-parameter family of infinite-dimensional diffusions on the
  {K}ingman simplex.
\newblock {\em Funktsional. Anal. i Prilozhen.}, 43(4):45--66, 2009.

\bibitem{MR2245368}
J.~Pitman.
\newblock {\em Combinatorial stochastic processes}, volume 1875 of {\em Lecture
  Notes in Mathematics}.
\newblock Springer-Verlag, Berlin, 2006.
\newblock Lectures from the 32nd Summer School on Probability Theory held in
  Saint-Flour, July 7--24, 2002, With a foreword by Jean Picard.

\bibitem{MR1337249}
Jim Pitman.
\newblock Exchangeable and partially exchangeable random partitions.
\newblock {\em Probab. Theory Related Fields}, 102(2):145--158, 1995.

\bibitem{MR3809477}
Jim Pitman and Yuri Yakubovich.
\newblock Ordered and size-biased frequencies in {GEM} and {G}ibbs' models for
  species sampling.
\newblock {\em Ann. Appl. Probab.}, 28(3):1793--1820, 2018.

\bibitem{MR0388547}
Frank Spitzer.
\newblock {\em Principles of random walk}.
\newblock Springer-Verlag, New York-Heidelberg, second edition, 1976.
\newblock Graduate Texts in Mathematics, Vol. 34.

\end{thebibliography}
\bibliographystyle{plain}
\nocite{*}
\vspace{1cm}
Department of Mathematics, University of California, Berkeley.\\
E-mail: \href{tdz@berkeley.edu}{tdz@berkeley.edu}

\end{document}